\documentclass[letterpaper,12pt,oneside,headsepline,DIV=10]{scrartcl}

\usepackage{etoolbox}
\usepackage{ifdraft}
\usepackage{iftex}

\ifluatex
\usepackage[utf8]{luainputenc}
\else
\usepackage[utf8]{inputenc}
\fi
\usepackage[T1]{fontenc}

\usepackage[english]{babel}

\usepackage{lmodern}
\usepackage{fourier}

\usepackage{amsfonts}
\usepackage{mathrsfs} %
\usepackage{dsfont}

\usepackage{amssymb}

\usepackage{amsmath}
\usepackage{mathtools}
\usepackage[thmmarks, amsmath, amsthm]{ntheorem}
\usepackage{nccmath}    %
\usepackage{tensor}

\usepackage[draft=false]{scrlayer-scrpage}
\usepackage{needspace}

\usepackage[backend=biber,style=numeric-comp,giveninits,url=false,sortcites=true,maxbibnames=99]{biblatex}
\mathtoolsset{showonlyrefs}

\usepackage[final,
            pdfborder={0 0 0}, colorlinks=true,
            linkcolor=BrickRed, citecolor=ForestGreen, urlcolor=RoyalBlue]{hyperref}

\usepackage[cmyk,dvipsnames]{xcolor}

\ifdraft{%
\usepackage[textsize=scriptsize,colorinlistoftodos]{todonotes}
\usepackage{lineno}
\usepackage{scrtime}
}{}

\usepackage{booktabs}
\usepackage[inline]{enumitem}
\usepackage{lettrine}
\usepackage[final]{listings}
\usepackage{url}

\ifdraft{%
\linenumbers
\newcommand*\patchAmsMathEnvironmentForLineno[1]{%
  \expandafter\let\csname old#1\expandafter\endcsname\csname #1\endcsname
  \expandafter\let\csname oldend#1\expandafter\endcsname\csname end#1\endcsname
  \renewenvironment{#1}%
     {\linenomath\csname old#1\endcsname}%
     {\csname oldend#1\endcsname\endlinenomath}}%
\newcommand*\patchBothAmsMathEnvironmentsForLineno[1]{%
  \patchAmsMathEnvironmentForLineno{#1}%
  \patchAmsMathEnvironmentForLineno{#1*}}%
\AtBeginDocument{%
\patchBothAmsMathEnvironmentsForLineno{equation}%
\patchBothAmsMathEnvironmentsForLineno{align}%
\patchBothAmsMathEnvironmentsForLineno{flalign}%
\patchBothAmsMathEnvironmentsForLineno{alignat}%
\patchBothAmsMathEnvironmentsForLineno{gather}%
\patchBothAmsMathEnvironmentsForLineno{multline}%
}}

\clearmainofpairofpagestyles
\automark[section]{subsection}

\renewcommand{\subsectionmark}[1]{}

\lohead{{\small\normalfont \headertitle}}
\rohead{{\small \headerauthors}}

\RedeclareSectionCommand[%
  font=\Large\sffamily\bfseries,%
  beforeskip=1\baselineskip,%
  afterskip=0.5\baselineskip,%
  indent=0em,%
  afterindent=false,%
  tocbeforeskip=0.3\baselineskip plus 1pt minus 1pt%
  ]{section}

\RedeclareSectionCommands[%
  font=\normalfont\bfseries,%
  beforeskip=3pt,%
  afterskip=-1em%
  ]{subsection,subsubsection}

\RedeclareSectionCommands[%
  font=\normalfont\itshape,%
  beforeskip=1pt,%
  afterskip=-1em,%
  indent=0pt,%
  ]{paragraph}

\AtEveryBibitem{\clearfield{doi}}
\AtEveryBibitem{\clearfield{isbn}}
\AtEveryBibitem{\clearfield{issn}}
\AtEveryBibitem{\clearfield{pages}}
\AtEveryBibitem{\clearlist{language}}

\renewcommand*{\bibfont}{\footnotesize}
\renewbibmacro*{in:}{}
\DeclareFieldFormat
  [article,inbook,incollection,inproceedings,patent,thesis,unpublished,misc]
  {title}{#1}

\setitemize{itemsep=0.02\baselineskip}

\newenvironment{enumeratearabic}{
\begin{enumerate}[label=(\arabic*),%
  leftmargin=2.5em,itemindent=0pt,%
  labelindent=.5em,labelwidth=1.5em,labelsep=!,%
  noitemsep]
}{
\end{enumerate}
}

\newenvironment{enumeratearabic*}{
\begin{enumerate*}[label=(\arabic*)] %
}{
\end{enumerate*}
}

\newenvironment{enumerateroman*}{
\begin{enumerate*}[label=(\roman*)] %
}{
\end{enumerate*}
}

\numberwithin{equation}{section}

\theoremnumbering{arabic}
\newtheorem{theoremcounter}{theoremcounter}[section]
\theoremnumbering{Alph}
\newtheorem{maintheoremcounter}{maintheoremcounter}

\theoremstyle{plain}

\newtheorem{lemma}[theoremcounter]{Lemma}

\newtheorem{proposition}[theoremcounter]{Proposition}
\newtheorem{theorem}[theoremcounter]{Theorem}

\theoremstyle{plain}

\newtheorem{maincorollary}[maintheoremcounter]{Corollary}
\newtheorem{maintheorem}[maintheoremcounter]{Theorem}

\theoremstyle{definition}

\theoremstyle{remark}

\newtheorem{remark}[theoremcounter]{Remark}

\theoremstyle{nonumberremark}

\newtheorem{mainremark}{Remark}

\newenvironment{mainremarkenumerate}
{%
\mainremark
\enumeratearabic
}{%
\endenumeratearabic
\endmainremark
}%

\newcommand{\tx}{\text}
\newcommand{\nbd}{\nobreakdash-\hspace{0pt}}

\newcommand{\texpdf}[2]{\texorpdfstring{#1}{#2}}

\makeatletter
\newcommand{\writelabel}[1]{#1\def\@currentlabel{#1}}
\makeatother

\newcommand{\minwidthmathbox}[2]{%
  \mathmakebox[{\ifdim#1<\width\width\else#1\fi}]{#2}%
}

\newcommand{\tbf}{\bfseries}

\newcommand{\cM}{\ensuremath{\mathcal{M}}}

\newcommand{\rmd}{\ensuremath{\mathrm{d}}}

\newcommand{\rmM}{\ensuremath{\mathrm{M}}}

\newcommand{\rmU}{\ensuremath{\mathrm{U}}}

\newcommand{\td}{\tilde}
\newcommand{\wtd}{\widetilde}
\newcommand{\ov}{\overline}

\newcommand{\wht}{\widehat}

\newcommand{\defeq}{\mathrel{:=}}

\newcommand{\condsep}{\mathrel{\;:\;}}

\newcommand{\mrelspace}[1]{\mathrel{\mspace{#1}}}

\NewCommandCopy{\rightarroworig}{\rightarrow}
\renewcommand{\rightarrow}
  {\protect\relbar\mrelspace{-9.7mu}\rightarroworig}

\NewCommandCopy{\leftarroworig}{\leftarrow}
\renewcommand{\leftarrow}
  {\protect\leftarroworig\mrelspace{-9.7mu}\relbar}

\newcommand{\ra}{\rightarrow}

\newcommand{\isdiv}{\mathrel{\mid}}
\newcommand{\nisdiv}{\mathrel{\nmid}}

\renewcommand{\pmod}[1]{\;(\mathrm{mod}\, #1)}

\newcommand{\ZZ}{\ensuremath{\mathbb{Z}}}

\newcommand{\U}[1]{\ensuremath{\mathrm{U}_{#1}}}

\newcommand{\Bigisdiv}{\mathrel{\Big|}}
\newcommand{\Bigisdivexact}{\mathrel{\Big\|}}

\newcommand{\headertitle}{{%
Theta Cycles of Modular Forms Modulo~$p^2$
}}
\newcommand{\headerauthors}{%
  S.~Ahlgren,
  M.~Raum,
  O.~K.~Richter%
}
\title{%
  Theta Cycles of Modular Forms Modulo~$p^2$
}
\author{%
  Scott Ahlgren%
  \thanks{The author was partially supported by a grant from the Simons Foundation (\#963004 to Scott Ahlgren).}
  \and%
  Martin Raum%
  \thanks{The author was partially supported by Vetenskapsr\aa det Grant~2023-04217.}%
  \and%
  Olav K. Richter%
  \thanks{The author was partially supported by a grant from the Simons Foundation (\#835652 to Olav Richter).}
}
\ifdraft{\date{\today\ at\ \thistime}}{\date{}}

\begin{document}

\thispagestyle{scrplain}
\begingroup
\deffootnote[1em]{1.5em}{1em}{\thefootnotemark}
\maketitle
\endgroup

\vspace{-\baselineskip}

\begin{abstract}
\small
\noindent
{\tbf Abstract:}
The theta cycle of a modular form modulo a prime~$p\geq 5$ is well understood. By contrast, the theta cycle modulo a power of~$p$ is still mysterious and experimentally erratic. Here we completely determine the theta cycle of a weight~$k < p$ modular form modulo~$p^2$ on the initial segment of length~$p$ and we prove exact values or nontrivial bounds for the weight filtrations on~$p-2$  further segments of length~$p - k + 1$. In particular, asymptotically as~$p \ra \infty$ we establish~50\% of the theta cycle exactly, and we provide nontrivial bounds for~100\% of it. We determine the first two low points exactly and~$\left\lfloor \frac{p-k+1}{2} \right\rfloor$ further low points at regular positions. Moreover, we detect low points at \emph{exceptional positions} which solve a quadratic equation modulo~$p$, and which disturb the otherwise regular structure in the segments that we exhibit.
\\[.3\baselineskip]
\noindent
\textsf{\textbf{%
  theta operator%
}}%
\noindent
\ {\tiny$\blacksquare$}\ %
\textsf{\textbf{%
  theta cycle%
}}%
\noindent
\ {\tiny$\blacksquare$}\ %
\textsf{\textbf{%
  low points%
}}
\\[.2\baselineskip]
\noindent
\textsf{\textbf{%
  MSC Primary: 11F33%
}}%
\ {\tiny$\blacksquare$}\ %
\textsf{\textbf{%
  MSC Secondary: 11F11%
}}
\end{abstract}

\Needspace*{4em}
\phantomsection
\label{sec:introduction}
\addcontentsline{toc}{section}{Introduction}
\markright{Introduction}

\lettrine[lines=2,nindent=.2em]{\tbf T}{\,hroughout}%
, let $p \ge 5$ be a prime.  Given a quasi-modular form~$f$ on $\operatorname{SL}_2(\ZZ)$ with $p$\nbd{}integral rational coefficients and an integer $m\geq 1$,  let $\omega_{p^m}(f)$ be the weight filtration of $f$ modulo~$p^m$ as defined in~\eqref{eq:def:weight_filtration}. Iteration of the theta operator~$\theta$ (defined in~\eqref{eq:def:theta_operator}) yields the extended \emph{theta cycle of~$f$ modulo~$p^m$}:
\begin{gather*}
  \Omega_{p^m} f
\defeq
  \Big(
  \omega_{p^m}\bigl( f\bigr),\,
  \omega_{p^m}\bigl(\theta^{1} f\bigr),\,
  \;\ldots\;,\,
  \omega_{p^m}\bigl(\theta^{(p-1)p^{m-1} + m-1} f\bigr)
  \Big)
\tx{,}
\end{gather*}
and the periodic subsequence starting with the weight filtration of~$\theta^m f$ is called the (Tate) \emph{theta cycle}. In the case~$m = 1$, theta cycles of modular forms are completely understood. As an example of a central application which leverages this detailed understanding we mention Edixhoven's work~\cite{edixhoven-1992}  on the weight in Serre's conjecture on modular forms.
In a different direction, this understanding leads to the classification~\cite{ahlgren-boylan-2003}  of Ramanujan congruences for the partition function.

An integer~$i$, or by extension the quasi-modular form~$\theta^i f $, is a \emph{low point} of the cycle if
\begin{gather*}
  \omega_{p^m}\big( \theta^{i-1} f \big)
>
  \omega_{p^m}\big( \theta^{i} f \big)
\quad\tx{and}\quad
  \omega_{p^m}\big( \theta^{i}f \big)
<
  \omega_{p^m}\big( \theta^{i+1}f \big)
\tx{.}
\end{gather*}
Further, we say that we have a \emph{rise} (respectively a \emph{fall}) at $i$ if
\begin{gather*}
  \omega_{p^m}\big( \theta^{i}f \big)
<
  \omega_{p^m}\big( \theta^{i+1}f \big)
  \qquad \text{\Big(respectively}
\quad
  \omega_{p^m}\big( \theta^{i}f \big)
>
  \omega_{p^m}\big( \theta^{i+1}f \big)
\tx{\Big).}
\end{gather*}
The theta cycles modulo~$p$ are determined by the positions and weight filtrations of their low points (Jochnowitz~\cite{jochnowitz-1982a}).
There are either one or two low points in a theta cycle modulo~$p$ and the remaining part of~$\Omega_p(f)$ is highly regular, as illustrated on the left in Figure~\ref{fig:k12_mod17_mod17sq}.
The behavior of the theta cycle modulo~$p^2$ is by contrast more intricate and seemingly erratic, as illustrated on the right.
Indeed, only partial results on the isolated points $i = 1$ (Chen--Kiming~\cite{chen-kiming-2016}) and $i \in p \ZZ$ or~$i \in p-k+1 + p\ZZ$ (Kim-Lee \cite{kim-lee-2023}) are available. In particular, not a single location of a low point is known.

As one example of the complicated structure which appears, we mention that the theta cycle modulo~$p^2$ can have successive falls; this is not possible in the cycle modulo~$p$.
An example of this phenomenon occurs for the normalized
weight $12$ cusp form $\Delta$ when $p=13$: we have
\begin{gather*}
  \omega_{13^2}\bigl( \theta^{133}\, \Delta \bigr) = 434
  \tx{,}\quad
  \omega_{13^2}\bigl( \theta^{134}\, \Delta \bigr) = 280
  \tx{,}\quad \tx{and} \quad
  \omega_{13^2}\bigl( \theta^{135}\, \Delta \bigr) = 126
  \tx{.}
\end{gather*}

\begin{figure}[ht]
\includegraphics[draft=false,width=0.49\linewidth]{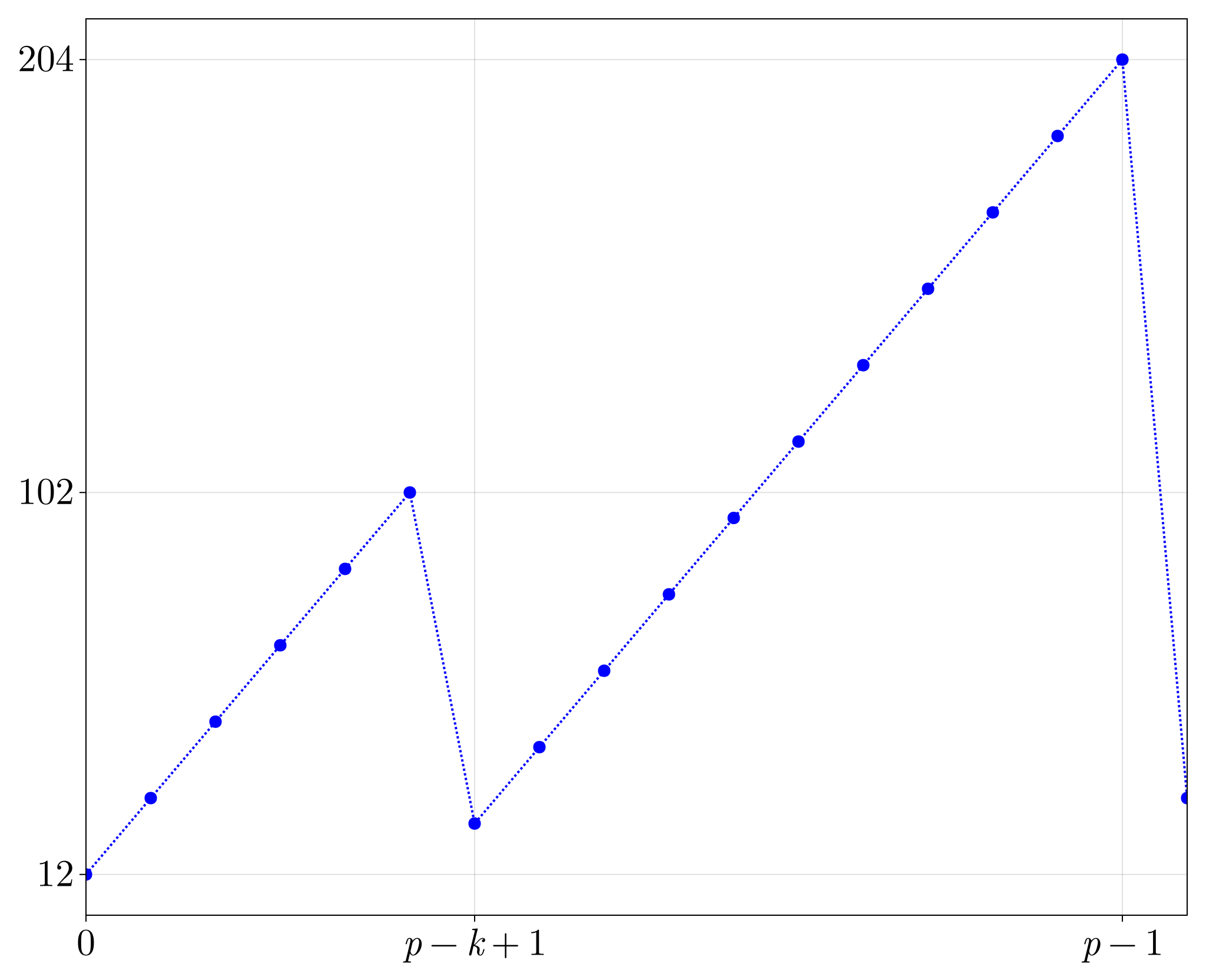}
\hfill
\includegraphics[draft=false,width=0.49\linewidth]{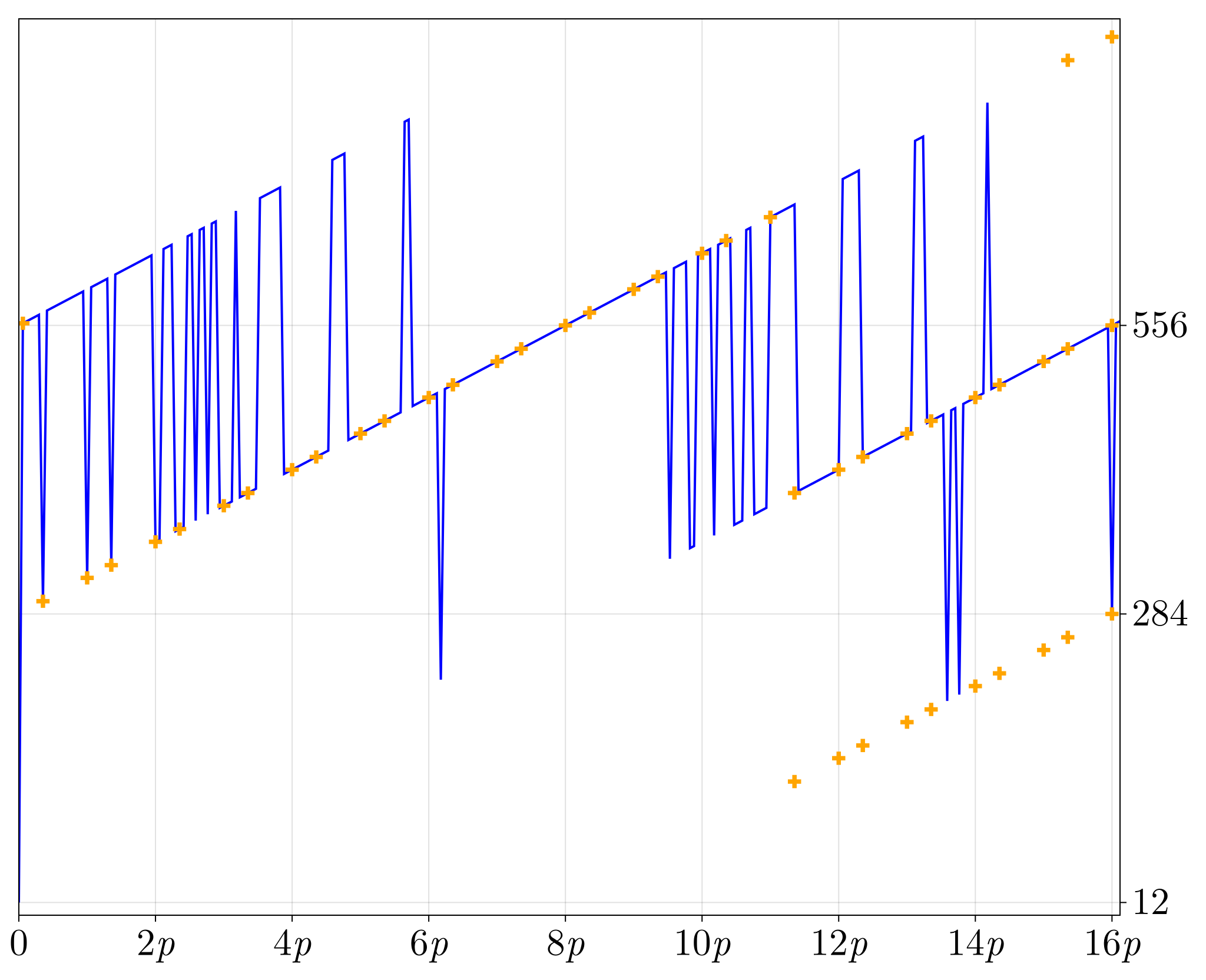}
\caption{
The weight filtrations modulo~$p = 17$ (on the left) and~$p^2$ (on the right) of~$\theta^i \Delta $, where~$\Delta$ is the normalized cusp form of weight~$k = 12$. Filtration values modulo~$p$ given for~$0 \le i \le p$ on the $x$-axis are connected by a dashed line. Filtrations modulo~$p^2$ given for~$0 \le i \le p (p-1)$ are represented by a blue line connecting the values. Orange crosses indicate previously known values (in some cases not uniquely determined) for~$i = 1$ \cite{chen-kiming-2016},
$i \in p \ZZ$ \cite{kim-lee-2023}, and~$i \in p - k + 1 + p \ZZ$ \cite{kim-lee-2023}.
}
\label{fig:k12_mod17_mod17sq}
\end{figure}

Here we determine exact values for weight filtrations in long segments of theta cycles modulo~$p^2$.  Our first result provides the exact weight filtration for all~$\theta^if$ with~$0 \le i \le p$.
A closer inspection of Figure~\ref{fig:k12_mod17_mod17sq} reveals the alignment of the first two low points in the theta cycles modulo~$p$ and~$p^2$.
Our theorems in particular imply that these low points always occur
and that they are the only ones in this range.

To state the result recall that~$p\geq 5$ and denote by $\rmM_k$ the space of weight $k$ modular forms on $\operatorname{SL}_2(\ZZ)$ with
$p$\nbd{}integral rational Fourier coefficients (recall that $\rmM_k=\{0\}$ unless $k=0$ or $k\geq 4$  is even).
We will often assume that $f\in \rmM_k$ with~$0<k<p$ has $\omega_p(f)=k$.  This assumption simply ensures that $f$ is not a constant multiple of the Eisenstein series $E_{p-1}\equiv 1\pmod p$.

\Needspace*{5\baselineskip}
\begin{maintheorem}
\label{mainthm:first_p_interval_weight_filtration}
Suppose that $f\in \rmM_k$ with~$0<k<p$ has $\omega_p(f)=k$. Then we have the following exact weight filtrations:
\begin{alignat*}{2}
  \omega_{p^2}\big( \theta^i f \big)
&=
  k
\tx{,}\quad
&&
  \tx{if\/ } i = 0
\tx{;}
\\
  \omega_{p^2}\big( \theta^i f \big)
&=
  k + 2i + 2p(p-1)
\tx{,}\quad
&&
  \tx{if\/ } 0 < i < p - k + 1
\tx{;}
\\
  \omega_{p^2}\big( \theta^i\, f \big)
&=
  k + 2i + p(p-1)
\tx{,}\quad
&&
  \tx{if\/ } i = p - k + 1
\tx{;}
\\
  \omega_{p^2}\big( \theta^i\, f \big)
&=
  k + 2i + 2p(p-1)
\tx{,}\quad
&&
  \tx{if\/ } p - k + 1 < i < p
\tx{;}
\\
  \omega_{p^2}\big( \theta^i\, f \big)
&=
  k + 2 i + p (p-1)
\tx{,}\quad
&&
  \tx{if\/ } i = p
\tx{.}
\end{alignat*}
\end{maintheorem}

The next corollary will follow from Theorem~\ref{mainthm:first_p_interval_weight_filtration} and Theorem~\ref{mainthm:first_half_of_p_intervals_weight_filtration}.

\Needspace*{5\baselineskip}
\begin{maincorollary}
\label{maincor:first_p_interval_weight_filtration_low_points}
Suppose that $f\in \rmM_k$ with~$0<k<p$ has $\omega_p(f)=k$. The first low point of the theta cycle of\/~$f$ modulo~$p^2$ occurs at\/~$i = p - k + 1$ and the second at\/~$i = p$. Moreover, if\/~$0 < i < p-1$  then
\begin{gather*}
  \omega_{p^2}\big( \theta^{i+1}\, f \big)
=\omega_{p^2}\big( \theta^{i}\, f \big) + 2
\quad
  \text{unless $i=p-k$}
\tx{.}
\end{gather*}
\end{maincorollary}

\begin{mainremarkenumerate}
\item
The weight filtrations at the first and second low points $i=p-k+1$, $i=p$ were determined in~\cite{kim-lee-2023}, but it was not proved that these are low points.

\item
The positions of the first low point modulo~$p$ and~$p^2$ agree, but unlike the modulo~$p$ case, the modulo~$p^2$ weight filtration of the first low point is independent of whether or not~$f$ has a~$\rmU_p$\nbd{}congruence modulo~$p$.

\item
The position of the second low point modulo~$p^2$ matches the position of the second low point for modular forms which do not have a~$\rmU_p$\nbd{}congruence modulo~$p$.

\item
The analysis in~\cite{jochnowitz-1982a} shows that $\omega_p(\theta^{i+1} f) \ne \omega_p(\theta^{i} f) + 2$ for all~$i \ge 1$, which is in contrast to the situation modulo~$p^2$, where a rise of $2$ is the common case.
\end{mainremarkenumerate}

Our next result determines a substantial part of the weight filtration 
modulo $p^2$ of~$\theta^if$ for general~$i$, 
and identifies families of low points.
Some of them appear at regular intervals, and some appear at what we call exceptional positions.
We will consider indices $i$ with 
\begin{gather}\label{eq:idef_Thm_C}
 np \le i \le np+p - k + 1\tx{,} \qquad 1 \le n < p
\tx{.}
\end{gather}
We call such an index \emph{exceptional} if it is a solution of the congruence
\begin{gather}\label{eq:def_exceptional}
   i^2 + (k-1) i - n^2 \equiv 0 \quad\pmod{p} 
   \tx{.}
\end{gather}
In Corollary~\ref{maincor:first_half_of_p_intervals_weight_filtration_low_points} we will show that 
these exceptional indices often correspond to low points which disrupt regularity in the theta cycle.

\Needspace*{5\baselineskip}
\begin{maintheorem}
\label{mainthm:first_half_of_p_intervals_weight_filtration}
Suppose that\/~$f\in \rmM_k$ with $0<k<p$ has~$\omega_p(f)=k$. Let\/~$i$ be as in \eqref{eq:idef_Thm_C}. 
\begin{enumeratearabic}
\item
 If\/~$i$ is exceptional 
 then
\begin{gather*}
  \omega_{p^2}\big( \theta^i\, f \big)
\le
  k + 2i + p (p-1)
\tx{.}
\end{gather*}

\item If\/~$i$ is not exceptional
then we have the following exact values and nontrivial bounds, where~$i' = i - np$ is the least non-negative residue of\/ $i\pmod p$.
\begin{alignat*}{2}
  \omega_{p^2}\big( \theta^i\, f \big)
&\le
  k + 2i + p (p-1)
\tx{,}\quad
&&
  \tx{if\/ } i' = 0
  \tx{ (proved in~\cite{kim-lee-2023})}
\tx{;}
\\
  \omega_{p^2}\big( \theta^i\, f \big)
&=
  k + 2i + p (p-1)
\tx{,}\quad
&&
  \tx{if\/ } 
  0 < i' < n
  \tx{ and\/ }
  i' \le p - k + 1 - n
\tx{;}
\\
  \omega_{p^2}\big( \theta^i\, f \big)
&=
  k + 2i + 2p (p-1)
\tx{,}\quad
&&
  \tx{if\/ } 
  n \le i' \le p - k + 1 - n
\tx{;}
\\
  \omega_{p^2}\big( \theta^i\, f \big)
&\le
  k + 2i + p (p-1)
\tx{,}\quad
&&
  \tx{if\/ } 
  p - k + 1 - n < i' < p - k + 1
\tx{;}
\\
  \omega_{p^2}\big( \theta^i\, f \big)
&\le
  k + 2i + p (p-1)
\tx{,}\quad
&&
  \tx{if\/ } i' = p - k + 1
  \tx{ (proved in~\cite{kim-lee-2023})}
\tx{.}
\end{alignat*}
\end{enumeratearabic}
\end{maintheorem}

For fixed weight and asymptotically as~$p \ra \infty$, our results provide exact filtration values for~50\% of the theta cycle modulo~$p^2$ and nontrivial bounds for~100\% of it (see Remark~\ref{rm:result_range}). We illustrate the range of our results for $f = \Delta$ and~$p = 59$ in Figure~\ref{fig:k12_mod59sq_results}.
\begin{figure}[h]
\includegraphics[draft=false,width=\linewidth]{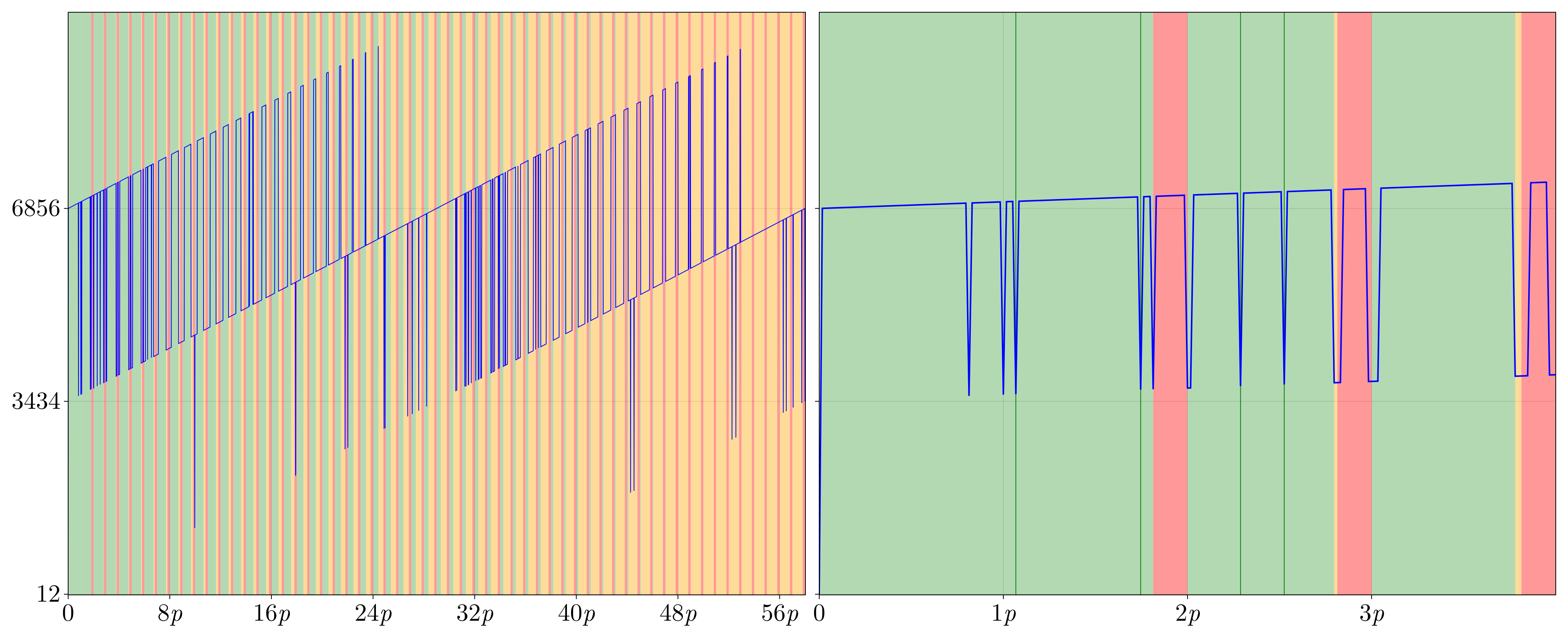}
\caption{The weight filtrations modulo~$p^2 = 59^2$ of~$\theta^i \Delta$, where~$\Delta$ is the normalized cusp form of weight~$12$. Filtration values are given for~$0 \le i \le p (p-1)$ (on the left) and~$0 \le i \le 4p$ (on the right) on the $x$-axis, and represented by a blue line connecting them. Vertical green lines on the right indicate exceptional low points (see Corollary~\ref{maincor:first_half_of_p_intervals_weight_filtration_low_points}). Green shaded areas represent ranges of~$i$ for which we prove exact filtration values. Orange shaded areas correspond to ranges of~$i$ for which we establish nontrivial upper bounds for the weight filtrations. We do not provide information on filtrations for~$i$ in the red shaded areas.}
\label{fig:k12_mod59sq_results}
\end{figure}
In this case, our results give exact filtration values for more than~33\% of the theta cycle modulo~$p^2$ and nontrivial bounds for more than~81.6\%. If $ f = \Delta$ and~$p > 1100$, then the range increases to more than~49\% of exact weight filtrations and more than~99\% of nontrivial filtration bounds.

In the next corollary we describe some of the structure of the theta cycle which is forced by 
Theorem~\ref{mainthm:first_half_of_p_intervals_weight_filtration}.  

\Needspace*{5\baselineskip}
\begin{maincorollary}
\label{maincor:first_half_of_p_intervals_weight_filtration_low_points}
Suppose that~$f\in \rmM_k$ with $0<k<p$ has~$\omega_p(f)=k$.

\begin{enumeratearabic}
\item
\label{it:maincor:first_half_of_p_intervals_weight_filtration_low_points:regular_low_points}
For~$1 \le n \le \frac{p - k + 1}{2}$ there is a low point at~$i = np + p - k + 2 - n$ unless we have $\omega_{p^2}( \theta^i\, f) = k+2i + p(p-1)$ and $\omega_{p^2}( \theta^{i+1}\, f)=k+2(i+1)$, in which case there is a low point at~$i+1$.

\item
\label{it:maincor:first_half_of_p_intervals_weight_filtration_low_points:exceptional_low_points}
Assume that~$np + n \le i < np + p - k + 1 - n$ and 
that $i$ is an exceptional index.
Then there is a low point at~$i$ (which we call an \emph{exceptional low point}).

\item
\label{it:maincor:first_half_of_p_intervals_weight_filtration_low_points:non_low_points}
Assume that~$np < i \le np + p - k + 1 - n$ and that~$i$ is not an exceptional index.
Then~$i$ is \emph{not} a low point.

\item
\label{it:maincor:first_half_of_p_intervals_weight_filtration_low_points:rises}
Assume that~$np < i < np + p - k + 1 - n$,  that~$i \ne np + n - 1$, and that
neither~$i$ nor~$i+1$ is an exceptional index.
Then there is a rise by~$2$ at~$i$.
\end{enumeratearabic}
\end{maincorollary}

\begin{mainremarkenumerate}
\item
In the range $0 \leq i < 2p - k + 1$
the two theorems give exact weight filtrations for~$\theta^if$ apart from possible exceptional indices.
 By part \ref{it:maincor:first_half_of_p_intervals_weight_filtration_low_points:exceptional_low_points} of Corollary~\ref{maincor:first_half_of_p_intervals_weight_filtration_low_points} and the fact that 
  the index~$i=2p-k$ is not exceptional it follows that there are exceptional low points at all exceptional indices in this range. 
\item
 There are two boundary cases $i'=0$ and $i'=p-k+1$ in Theorem~\ref{mainthm:first_half_of_p_intervals_weight_filtration};  these indices cannot be exceptional.
 These cases were studied in detail in~\cite{kim-lee-2023}, and in many cases exact values for the filtration are given in \cite[Theorem 1.8]{kim-lee-2023}. Here we give a bound which follows from this result and which applies to all cases.

\item
The exact filtration values in Theorem~\ref{mainthm:first_half_of_p_intervals_weight_filtration} for~$n \le i' \le p - k + 1 - n$ are independent of whether~$f$ has a~$\rmU_p$\nbd{}congruence or not. Thus, information on~$\rmU_p$\nbd{}congruences must be encoded in the part of the theta cycle for which we do not have exact filtration values.

\item
The filtration bounds in Theorem~\ref{mainthm:first_half_of_p_intervals_weight_filtration} are quadratic in~$p$, while the bounds in Lem\-ma~\ref{la:naive_upper_bounds_cycle_filtrations} and Remark~\ref{rm:la:naive_upper_bounds_cycle_filtrations} below are cubic in~$p$, and the bounds which follow from Inequality~\eqref{eq:trivial_upper_bounds_cycle_filtrations} are quartic in~$p$.

\item
Let~$f$ be a modular form as in Theorems~\ref{mainthm:first_p_interval_weight_filtration} and~\ref{mainthm:first_half_of_p_intervals_weight_filtration}. Then Corollary~\ref{maincor:first_p_interval_weight_filtration_low_points} identifies two low points occurring at regular positions, and the first part of Corollary~\ref{maincor:first_half_of_p_intervals_weight_filtration_low_points} identifies $\left\lfloor \frac{p - k + 1}{2} \right\rfloor$ such low points. All of these occur in the green-shaded areas of Figure~\ref{fig:k12_mod59sq_results}.

\item
It would be interesting to determine the number of exceptional low points, which would amount to counting solutions to the congruence~$i^2 + (k-1) i - n^2 \equiv 0 \,\pmod{p}$ with~$n \le i < p - k + 1 - n$. Heuristically, the number of such solutions should be proportional to~$p$ on average as~$p$ varies.
\end{mainremarkenumerate}

We briefly describe some of the new ideas and input which allow us to access a large part of the theta cycle.
In the next section we introduce the \emph{factor filtration},
a refinement of the weight filtration which plays a key role (all of the results above will follow from our study of the corresponding factor filtrations). Roughly speaking, the factor filtration 
involves dividing out as many powers of $E_{p-1}\equiv 1\pmod p$ as possible before considering the weight filtration.
This idea was inspired in part by recent work of the authors with Hanson~\cite{ahlgren-hanson-raum-richter-2025} in which we prove that every Eisenstein series on $\operatorname{SL}_2(\ZZ)$ has uniformly low factor filtration modulo $p^2$.  In the same paper we determine a precise expression \eqref{eq:e2_congruence} for $E_2\pmod{p^2}$; this is another important input here.
The starting point for our results is 
 the expansion \eqref{eq:thetai_e2_power_expansion} of~$\theta^if$ in terms of $E_2$.   Our ability to determine the precise factor filtration depends on isolating a group of terms in this expansion which combine to give a unique term of highest factor filtration (we often use the properties of the theta cycle modulo $p$  in order to isolate these terms).
When we are unable to isolate such a group of terms the method still provides a non-trivial bound for the factor filtration.  Finally, 
the weight filtrations are easily determined using Lemma~\ref{la:weight_from_factor_filtration}.
   
   After some preliminaries in 
Section~\ref{sec:preliminaries}, 
 we prove Theorem~\ref{mainthm:first_p_interval_weight_filtration} 
 in Section~\ref{sec:first_p_interval}.
The proofs of Theorem~\ref{mainthm:first_half_of_p_intervals_weight_filtration}, Corollary~\ref{maincor:first_p_interval_weight_filtration_low_points}, and Corollary~\ref{maincor:first_half_of_p_intervals_weight_filtration_low_points}
are more involved and occupy Section~\ref{sec:first_half_p_interval}.
In Section~\ref{ssec:preliminaries:upper_bounds_cycle_filtrations} we discuss the quality of the general bounds which are obtained for filtrations modulo $p^2$.

\subsection*{Acknowledgements}

We thank Amy Woodall and the referees for many helpful suggestions which improved our exposition.

\section{Preliminaries}%
\label{sec:preliminaries}

Let~$p \geq 5$ be prime, and let~$\rmM_k$ be the space of holomorphic weight~$k$ modular forms on $\operatorname{SL}_2(\ZZ)$ with~$p$\nbd{}integral coefficients.
If~$f \in \rmM_k$ it has a Fourier expansion~$\sum a(n) q^n$, and we identify it with this expansion.
For even~$k \ge 2$ let
\begin{gather*}
  E_k
\defeq
  1 - \mfrac{2k}{B_{2k}}\, \sum_{n=1}^\infty \sigma_{k-1}(n)\, q^n
\end{gather*}
be the weight~$k$ Eisenstein series, which for~$k = 2$ is quasi-modular and for~$k > 2$ modular.

Let $m$ be a positive integer, and let $\cM_{\bullet,m} \subseteq\left(\ZZ/p^m\ZZ\right)\llbracket q \rrbracket$ be the set of reductions modulo $p^m$ of all elements of all $\rmM_k$, i.e.\@ the reductions of their Fourier expansions. If $\ov f =\sum \ov{a(n)}q^n\in \cM_{\bullet,m}$  then we define the weight filtration
\begin{gather}
\label{eq:def:weight_filtration}
  \omega_{p^m}(\ov f)
\defeq
  \inf\big\{
  k
  \condsep
  \ov f \equiv g \,\pmod{p^m}
  \tx{ for some }
  g \in \rmM_k
  \big\}
\tx{.}
\end{gather}
Recall that
\begin{gather}
\label{eq:epm_congruences}
  E_{p-1}\equiv 1 \;\pmod p
\quad\tx{and therefore}\quad 
  E_{p-1}^{p^{m-1}} \equiv 1 \;\pmod{p^m}
\tx{.}
\end{gather}
As a refinement of the weight filtration, we introduce the factor filtration, which is defined 
by
\begin{gather}
\label{eq:def:factor_filtration}
  \wtd\omega_{p^m}(\ov f)
\defeq
  \inf\big\{
  k
  \condsep
  \ov f \equiv E_{p-1}^n\, g \,\pmod{p^m}
  \tx{ for some } n \ge 0
  \tx{ and some } g \in \rmM_k
  \big\}
\tx{.}
\end{gather}
By a standard abuse of notation we will write
\begin{gather*}
  \omega_{p^m}(f) = \omega_{p^m}(\ov f)
\tx{,}\qquad
  \wtd\omega_{p^m}(f) = \wtd\omega_{p^m}(\ov f)
\end{gather*}
when $f\in\ZZ_{(p)} \llbracket q \rrbracket$ has $\ov f=f\pmod {p^m}\in \cM_{\bullet,m}$.  
From~\cite[Lemma 5]{swinnerton-dyer-1973} we have
$\omega_p(f)=\wtd\omega_p(f)$.
If $f_1\in \rmM_{k_1}$, $f_2\in \rmM_{k_2}$ have $f_1, f_2 \not\equiv 0 \,\pmod p$, then by 
\cite[Theorem 1]{Serre-p-adic} we have 
\begin{gather}\label{eq:serremodpm}
  f_1 \equiv f_2 \;\pmod {p^m}
\implies
  k_1\equiv k_2 \;\pmod{p^{m-1}(p-1)}
\tx{.}
\end{gather}
It follows that if~$\ov f\in \cM_{\bullet,m}$ is the 
reduction modulo $p^m$ of a modular form of weight $k$, and if~$\ov f\not\equiv 0\pmod p$, then
\begin{gather}\label{eq:redmodpm}
  \omega_{p^m}(\ov f) \equiv k \;\pmod {p^{m-1}(p-1)}
\tx{.}
\end{gather}
We also have
\begin{gather}\label{eq:filt_eq_mod_p}
\omega_{p^m}(f) \equiv \wtd\omega_{p^m}(f) \;\pmod{p-1}
\tx{,}
\end{gather}
which follows from observing that~$\omega_{p^m}(\ov{f}) = \omega_{p^m}(E_{p-1}^n g)$ on the right hand side of~\eqref{eq:def:factor_filtration}.

The next lemma, which follows from the definitions and these facts, allows us to determine the weight filtration from the factor filtration. 
\begin{lemma}%
\label{la:weight_from_factor_filtration}
Suppose that $\ov f\in \cM_{\bullet,m}$ has $\ov f \not\equiv 0 \,\pmod{p}$ and that $\ov f$ is the reduction modulo $p^m$ of a modular form of weight $k$.  Then $\omega_{p^m}(\ov f)$ is the smallest integer such that 
\begin{gather}\label{eq:lemma1}
  \omega_{p^m}(\ov f) \geq \wtd\omega_{p^m}(\ov f)
\quad\text{and}\quad
  \omega_{p^m}(\ov f) \equiv k \;\pmod{p^{m-1}(p-1)}
\tx{.}
\end{gather}
\end{lemma}

\begin{proof}
Let $k_0=\wtd\omega_{p^m}(\ov f)$.  By definition there exists $g\in M_{k_0}$ such that $\ov f\equiv E_{p-1}^ng\,\pmod{p^m}$ for some $n\geq 0$.  Letting $n'$ be the least non-negative residue of $n$ modulo $p^{m-1}$, we see by~\eqref{eq:epm_congruences} that~$\ov f\equiv E_{p-1}^{n'}g\,\pmod{p^m}$;  we also see that~$n'(p-1)+k_0$ is the integer characterized by \eqref{eq:lemma1}.
If $\omega_{p^m}(\ov f)\neq n'(p-1)+k_0$ then by \eqref{eq:redmodpm} we would have the clear contradiction
 \begin{gather*}
     \omega_{p^m}(\ov f)\leq n'(p-1)+k_0-p^{m-1}(p-1)<k_0=\wtd\omega_{p^m}(\ov f)
     \tx{.}
     \end{gather*}
\end{proof}

We will often use the following bounds for the factor filtration.
First, for~$\ov f, \ov g\in \cM_{\bullet,m}$  it is clear from the definition that 
\begin{gather}
\label{eq:factfilt_prod}
  \wtd\omega_{p^m}(\ov f\, \ov g)
\leq
  \wtd\omega_{p^m}(\ov f) + \wtd\omega_{p^m}(\ov g)
\tx{.}
\end{gather}
Further, if~$\ov f$ and~$\ov g$ are the reductions of elements of $\rmM_k$ then we have
\begin{gather}
\label{eq:factfilt_nonarch}
  \wtd\omega_{p^m}(\ov f +\ov g)
\leq
  \max\big\{\wtd\omega_{p^m}(\ov f), \wtd\omega_{p^m}(\ov g)\big\}
\quad
  \text{with equality if }\ \wtd\omega_{p^m}(\ov f) \neq \wtd\omega_{p^m}(\ov g)
\tx{.}
\end{gather}
For the equality in \eqref{eq:factfilt_nonarch} we argue as follows: Suppose that $k_0=\wtd\omega_{p^m}(\ov f)>\wtd\omega_{p^m}(\ov g)$.
Then $\ov f+\ov g\equiv E_{p-1}^a\left(f'+E_{p-1}g'\right)\, \pmod{p^m}$ for some $f'\in M_{k_0}$, $g'\in M_{k_0-(p-1)}$.  If $\wtd\omega_{p^m}(\ov f+\ov g)<k_0$ then $f'+E_{p-1}g'\equiv E_{p-1}h\,\pmod{p^m}$ for some $h\in M_{k_0-(p-1)}$, which contradicts $k_0=\wtd\omega_{p^m}(\ov f)$.

\subsection{The theta operator and Serre derivative}%
\label{ssec:preliminaries:theta_serre}

If $f=\sum a(n)q^n\in \ZZ_{(p)}\llbracket q \rrbracket$  then 
the action of the 
 theta operator is given by 
\begin{gather}
\label{eq:def:theta_operator}
  \theta\, f
\defeq
  q
  \frac{\rmd f}{\rmd\! q}=
  \sum na(n)q^n
\tx{.}
\end{gather}
It does not preserve modularity, but it allows for a modular correction, which leads to the Serre derivative:  if $f\in \rmM_k$ then (since $p\geq 5$) we have
\begin{gather}
\label{eq:def:serre_derivative}
  \partial\, f
\defeq
  \partial_k\, f
\defeq
  \theta\, f
  -
  \mfrac{k}{12} f E_2\in \rmM_{k+2}
\tx{.}
\end{gather}
Note that there are two common normalizations;  the Serre derivative in~\cite{ahlgren-hanson-raum-richter-2025} corresponds to~$12 \partial$ in this paper.

From~(64) of the first part of~\cite{bruinier-van-der-geer-harder-zagier-2008}, we adopt the modified Serre derivative
\begin{gather}
\label{eq:def:modified_serre_derivative}
  \wht\partial_k^0\,f
\defeq
  f
\tx{,}\quad
  \wht\partial_k^1\,f
\defeq
  \partial_k\, f
\tx{,}\quad
  \wht\partial_k^{i+1}\,f
\defeq
  \partial_{k+2i} \big(
  \wht\partial_k^i\, f
  \big)
  -
  \mfrac{i (i+k-1)}{12^2} E_4\,
  \wht\partial_k^{i-1}\, f
\quad
  \tx{for } i \ge 1
\tx{.}
\end{gather}
Observe that our notation differs from that in~\cite{bruinier-van-der-geer-harder-zagier-2008} in order to accommodate established notation for the 
theta operator.
The modified Serre derivative preserves modularity
and~$p$\nbd{}integrality of Fourier coefficients: if~$f \in \rmM_k$, then we have
\begin{gather}
\label{eq:serre_derivative_weight}
  \wht\partial^i_k\, f
\in
  \rmM_{k + 2i}
\tx{.}
\end{gather}

The theta cycle modulo~$p$ was completely described by Jochnowitz (following Tate).
Here we say that $f$ is non-ordinary (at $p$) if it has a~$\rmU_p$\nbd{}congruence modulo~$p$, i.e.\@ $a(n p) \equiv 0 \,\pmod{p}$ for all~$n \in \ZZ$, which is equivalent to~$\theta^{p-1}\,f\equiv f\pmod p$. Note that the notion of ordinary and non-ordinary modular forms is usually reserved for eigenforms, and here we extend it to all modular forms.

\begin{proposition}%
[\cite{jochnowitz-1982a}]%
\label{prop:theta_cycle_modp}
Suppose that $f\in \rmM_k$ with~$0<k<p$ has
$\omega_p(f)=k$. 
Then 
\begin{gather*}
    \theta^i\, f \equiv \theta^{i-p+1}\, f \,\pmod{p}\quad \text{for $i\geq p$}\tx ,
    \end{gather*}
     and for $0 \le i < p$ we have the filtration values
\begin{gather*}
  \omega_p\big( \theta^i\, f \big)
=
\left\{
\begin{aligned}
&
  k + i (p+1)
\tx{,}\mspace{-4mu}
&&
  \tx{if\/ } 0 \le i < p - k + 1
\tx{;}
\\
&
  k + i (p+1) - (p-k+1)(p-1)
\tx{,}\mspace{-4mu}
&&
  \tx{if\/ } p - k + 1 \le i < p
  \tx{ and } f \tx{ is ordinary}
\tx{;}
\\
&
  k + i (p+1) - (p-k+2)(p-1)
\tx{,}\mspace{-4mu}
&&
  \tx{if\/ } p - k + 1 \le i < p-1
\\
&&&
  \tx{and } f \tx{ is non-ordinary}
\tx{;}
\\
&
  k=k + i (p+1) - (p+1)(p-1)
\tx{,}\mspace{-4mu}
&&
  \tx{if\/ } i=p-1
  \tx{ and } f \tx{ is non-ordinary}
\tx{.}
\end{aligned}
\right.
\end{gather*}
\end{proposition}
We recall two basic facts
\cite[Lemma 5]{swinnerton-dyer-1973}, \cite[Lemma 1]{Serre-p-adic}:  If $\ov f\in\cM_{\bullet,1}$ then 
\begin{align*}
  \omega_p(\ov f^a)
&=
  a\,\omega_p(\ov f)
\tx{,}
\\
  \omega_p(\theta\, \ov f)
&\leq
  \omega_p(\ov f)+p+1
\quad
  \tx{with equality unless } \omega_p(\ov f) \equiv 0 \,\pmod p
\tx{.}
\end{align*}

\subsection{Expansion of~\texpdf{$\theta^i\, f$}{theta i f} in~\texpdf{$E_2$}{E2}}%
\label{ssec:preliminaries:thetai_e2_power_expansion}

If~$f\in \rmM_k$, then the quasi-modular form $\theta^i\, f$ 
has an expansion in powers of the Eisenstein series~$E_2$. In particular, from~(65) of~\cite[Chapter~1]{bruinier-van-der-geer-harder-zagier-2008} we have the expansion
\begin{gather}
\label{eq:thetai_e2_power_expansion}
  \theta^i\, f
=
  \sum_{j = 0}^i
  \mbinom{i}{j}\, \mfrac{(i+k-1)!}{(j+k-1)!}\,
  \big( \wht\partial_k^j\,f \big) \, \big(\tfrac{1}{12} E_2 \big)^{i-j}
\tx{,}
\end{gather}
where~$\wht\partial_k$ is the modified Serre derivative in~\eqref{eq:def:modified_serre_derivative}. For simplicity, we will write
\begin{gather*}
  f_j
\defeq
  \wht\partial_k^j\,f
\end{gather*}
throughout the paper when using this formula.

\subsection{The weight-\texpdf{$2$}{2} Eisenstein series modulo \texpdf{$p^2$}{p\textasciicircum 2}}%
\label{ssec:preliminaries:e2}

For the duration we will focus on the case $m=2$.
From \cite{swinnerton-dyer-1973} (see Theorem~2 (iii) and Lemma~5) we know 
 that 
\begin{gather}\label{eq:e2powermodp}
 \omega_p(fE_2^n)= \omega_p(fE_{p+1}^n)=\omega_p(f)+n(p+1),\quad n\geq 0\tx{;}
\end{gather}
to make full use of~\eqref{eq:thetai_e2_power_expansion} we need the analogue of this fact for factor filtrations modulo $p^2$.
From Theorem~1.2 of \cite{ahlgren-hanson-raum-richter-2025}  and the remarks which follow (recall that the Serre derivative is normalized differently there) we have
\begin{gather}
\label{eq:e2_congruence}
  E_2
\equiv
  12 \partial E_{p-1}\, E_{p-1}^{2p-1} + p E_{p+1}^{p}\, E_{p-1}^{p-2} 
  \quad\pmod{p^2}
\tx{}
\end{gather}
(note that the right side is a modular form of weight~$2+2 p (p-1)$).
Since
\begin{gather}\label{eq:delta_e2_cong}
    12 \partial E_{p-1} \equiv E_{p+1} \quad\pmod{p}\tx{,}
\end{gather}
it follows that 
\begin{gather}
\label{eq:e2_power_congruence}
   E_2^n
\equiv
  \big( 12 \partial E_{p-1} \big)^n\,
  E_{p-1}^{n (2p-1)}
  +
  p n E_{p+1}^{n+p-1}\,
  E_{p-1}^{(n-1)(2p-1)+p-2}
  \quad\pmod{p^2}
\tx{.}
\end{gather}
We will need the following result; with $f=1$ and $k=0$ it gives the factor filtration of $E_2^n$.

\begin{lemma}\label{lem:E2_powers}
Suppose that~$k\geq 0$,   that $\omega_p(f)=k$, and that $n\geq 0$.
Then

\begin{align}
\wtd\omega_{p^2}(pfE_2^n)
&=
  k+2n+n(p-1).
\\
\intertext{If in addition we have~$\wtd\omega_{p^2}(f)=k$ then}
  \wtd\omega_{p^2}(fE_2^n)
&=
\begin{cases}
  k+2n+n (p-1) + (p+1) (p-1)
  \quad
  &
  \tx{if } p \nisdiv n
  \tx{;}
  \\
  k+2n+n (p-1)
  \quad
  &
  \tx{if } p \isdiv n
  \tx{.}
\end{cases}
\end{align}

\end{lemma}
\begin{proof}
Note that for any $f$ we have
\begin{gather} 
\label{eq:filt_omegaf}
  \wtd\omega_{p^2}(f) \geq \omega_p(f)
\quad\tx{and}\quad
  \wtd\omega_{p^2}(pf) = \omega_p(f)
\tx{.}
\end{gather}
For~$n > 0$, \eqref{eq:e2_power_congruence} gives
\begin{gather}
\label{eq:fe2_power_congruence}
  fE_2^n
\equiv
  f\big( 12 \partial E_{p-1} \big)^n\,
  E_{p-1}^{n (2p-1)}
  +
  p nf E_{p+1}^{n+p-1}\,
  E_{p-1}^{(n-1)(2p-1)+p-2}
  \quad\pmod{p^2}
\tx{.}
\end{gather}
If $p\mid n$ then \eqref{eq:fe2_power_congruence} gives $\wtd\omega_{p^2}(fE_2^n)\leq k+n(p+1)$, 
while \eqref{eq:filt_omegaf} and \eqref{eq:e2powermodp} give
\begin{gather*}
\wtd\omega_{p^2}(fE_2^n)\geq\omega_p(fE_2^n)=k+n(p+1)=k+2n+n(p-1)\tx{,}
\end{gather*}
so the lemma follows in this case.

Suppose that~$p\nmid n$.
The factor filtration of the first term on the right side of \eqref{eq:fe2_power_congruence} is at most $k+n(p+1)$, while by \eqref{eq:filt_omegaf} and \eqref{eq:e2powermodp} the factor filtration of the second term is 
\begin{gather*}
\omega_p(fE_2^{n+p-1})=k+(n+p-1)(p+1)=k+2n+n(p-1)+(p+1)(p-1)\tx{.}
\end{gather*}
By \eqref{eq:factfilt_nonarch}  
we see that
\begin{gather*}
\wtd\omega_{p^2}(fE_2^n)=k+2n+n(p-1)+(p+1)(p-1)\tx{.}
\end{gather*}
Finally, 
if $\omega_p(f)=k \ge 0$ then
\begin{gather*}
\wtd\omega_{p^2}(pfE_2^n)=\omega_p(fE_2^n)=k+n(p+1)\tx{.}
\end{gather*}
\end{proof}

\subsection{First upper bounds for the theta cycle filtrations}%
\label{ssec:preliminaries:upper_bounds_cycle_filtrations}
Suppose that $f\in \rmM_k$ with $0<k<p$ has $\omega_p(f)=k$. From Proposition~\ref{prop:theta_cycle_modp} we see that $\theta^i\, f\not\equiv 0\pmod p$ for any $i$.  
Relation~\eqref{eq:def:serre_derivative} together with~\eqref{eq:redmodpm} and the expression \eqref{eq:e2_congruence} for $E_2$ yield 
\begin{gather}
\label{eq:filtcong}
  \omega_{p^2}\big( \theta^i f \big)
\equiv
  k+2i
  \quad\pmod{p(p-1)}
\tx{.}
\end{gather}
Furthermore, \eqref{eq:def:serre_derivative} 
and \eqref{eq:e2_congruence}  imply that if $f\in \rmM_k$, then
\begin{gather}
\label{eq:trivial_upper_bounds_cycle_filtrations}
  \wtd\omega_{p^2}\big( \theta^i\, f \big)
\le
  \omega_{p^2}\big( \theta^i\, f \big)
\le
  k + 2 i
  +
  2 i p (p-1)
\tx{.}
\end{gather}
This also follows from work of Chen--Kiming~\cite[Theorem~1]{chen-kiming-2016} on the theta operator modulo~$p^m$ for general~$m$. The bound~\eqref{eq:trivial_upper_bounds_cycle_filtrations}  is however far from optimal, and Lemma~\ref{lem:E2_powers} already yields much improved bounds.

\begin{lemma}%
\label{la:naive_upper_bounds_cycle_filtrations}
Assume that~$f \in \rmM_k$. Then we have the following bounds on the factor and weight filtrations:
\begin{gather}
\label{eq:naive_upper_bounds_cycle_filtrations}
\begin{aligned}
  \wtd\omega_{p^2}\big( \theta^i\, f \big)
&\le
  k + 2i
  +
  (i + p + 1) (p-1)
\tx{,}
\\
  \omega_{p^2}\big( \theta^i\, f \big)
&\le
  k + 2i
  +
  \big( \big\lceil \tfrac{i + 1}{p} \big\rceil + 1 \big) p (p-1)
\tx{.}
\end{aligned}
\end{gather}
\end{lemma}

\begin{proof}
We estimate the factor filtration of each term in the expansion of~$\theta^i\, f$ given in~\eqref{eq:thetai_e2_power_expansion}. 
By~\eqref{eq:serre_derivative_weight} we have $\wtd\omega_{p^2}(f_j)\leq k+2j$.
From Lemma~\ref{lem:E2_powers} we obtain
\begin{align*}
  \wtd\omega_{p^2}\Big(
  \big( \wht\partial_k^j\, f \big)
  \big( \tfrac{1}{12} E_2 \big)^{i - j}
  \Big)
&\le
  k + 2j
  \,+\,
  2(i-j) + (i-j + p + 1) (p-1)
\\
&=
  k + 2i
  +
  (i-j + p + 1) (p-1)
\tx{,}
\end{align*}
and the stated bound for the factor filtration follows from \eqref{eq:factfilt_nonarch}. The weight filtration is determined using Lemma~\ref{la:weight_from_factor_filtration}  and \eqref{eq:filtcong}.
\end{proof}

\begin{remark}\label{rm:la:naive_upper_bounds_cycle_filtrations}
The bounds in Lemma~\ref{la:naive_upper_bounds_cycle_filtrations} are monotone in~$i$. 
Since
\begin{gather*}
    \theta^{i + p(p-1)}\, f \equiv \theta^i\, f \,\pmod{p^2}\quad\text{ for~$i \ge 2$,}
\end{gather*}    
we immediately obtain upper bounds for all $i$ which are cubic in $p$:
\begin{align*}
  \wtd\omega_{p^2}\big( \theta^i\, f \big)
&\le
  k + 2 (p (p-1) + 1)
  +
  (p^2 + 2) (p-1)
\tx{,}
\\
  \omega_{p^2}\big( \theta^i\, f \big)
&\le
  k + 2 (p (p-1) + 1)
  +
  (p+1) p (p-1)
\tx{.}
\end{align*}
This improves the upper bound~\eqref{eq:trivial_upper_bounds_cycle_filtrations}, which is quartic in~$p$ if~$i \asymp p^2$. Our results replace this by a quadratic upper bound for asymptotically $100\%$ of the theta cycle
(see Remark~\ref{rm:result_range} below).
\end{remark}

\section{Filtrations in the first \texpdf{$p$}{p}-interval}%
\label{sec:first_p_interval}

In this section, we prove Theorem~\ref{mainthm:first_p_interval_weight_filtration}, which will follow from a series of propositions which determine the factor filtrations for all~$\theta^i f $ with~$0 \le i \le p$.

\subsection{Filtrations up to the first low point}%
\label{ssec:first_p_interval:first_low_point}

Recall that the first low point of the theta cycle modulo $p$ occurs at $i=p-k+1$ (we will show that the same is true modulo $p^2$).  

\begin{proposition}%
\label{prop:before_first_low_point_filtrations}
Suppose that $f\in \rmM_k$ with~$0<k<p$ has $\omega_p(f)=k$. In the range~$0 < i < p - k + 1$  we have the factor filtration
\begin{gather*}
  \wtd\omega_{p^2}\big( \theta^i f \big)
=
  k + 2i + (i + p + 1) (p - 1)\tx{.}
\end{gather*}
\end{proposition}

\begin{proof}
We investigate the factor filtrations of the terms in~\eqref{eq:thetai_e2_power_expansion}. From Lemma~\ref{lem:E2_powers} and the fact that~$f_0 = f$ we have 
\begin{gather}
\label{eq:filt_f0_E2}
  \wtd\omega_{p^2}\big( f_0 E_2^i \big)
=
  k + 2i + (i + p + 1) (p-1)
\tx{.}
\end{gather}
Since~$\wtd\omega_{p^2}(f_j) \le k + 2j$, the factor filtrations of the terms in~\eqref{eq:thetai_e2_power_expansion} with~$j> 0$ are bounded by
\begin{gather*}
\wtd\omega_{p^2}\big( f_j\, E_2^{i-j} \big)
\le
  k + 2 j + 2(i-j) + (i-j+p+1) (p-1)
=
  k + 2 i + (i-j+p+1) (p-1)
\tx{.}
\end{gather*}
Therefore the factor filtration of~\eqref{eq:thetai_e2_power_expansion} is dominated by the term \eqref{eq:filt_f0_E2}  and the proposition follows from \eqref{eq:factfilt_nonarch}.
\end{proof}

\begin{proposition}%
\label{prop:first_low_point_filtration}
Suppose that $f\in \rmM_k$ with~$0<k<p$ has
$\omega_p(f)=k$. 
Then at~$i = p - k + 1$ we have the factor filtration
\begin{gather*}
  \wtd\omega_{p^2}\big( \theta^i\, f \big)
=
  k + 2 i + i (p-1)
=
  k + i (p+1)
\tx{.}
\end{gather*}
\end{proposition}

\begin{proof}
If $j<i$ then (since $i+k-1=p$) we have the following bound for the $j$-th term in~\eqref{eq:thetai_e2_power_expansion}:
\begin{gather*}
  \wtd\omega_{p^2} \Big( \mbinom{i}{j}\, \mfrac{(i+k-1)!}{(j+k-1)!}\, f_j\, E_2^{i-j} \Big)
=
  \omega_p \big(f_j\, E_2^{i-j} \big)
\le
  k + 2 j
  +
  (i-j) (p+1)
=
  k + 2 i
  +
  (i-j) (p-1)
\tx{.}
\end{gather*}
For~$j = 0$ this is an equality by~\eqref{eq:e2powermodp}.
On the other hand, the term with~$j = i$ in~\eqref{eq:thetai_e2_power_expansion} gives
\begin{gather*}
  \wtd\omega_{p^2}( f_i )
\le
  k + 2 i
\tx{.}
\end{gather*}
Thus the factor filtration of $\theta^i f$ is determined again by the term with~$j = 0$.
\end{proof}

\subsection{Filtrations between the first and second low point}%
\label{ssec:first_p_interval:second_low_point}

Here we determine the factor filtration of $\theta^i f$ for $p-k+1 < i\leq p$.
(We will show that there is a second low point of the theta cycle at $i=p$.)
We again use \eqref{eq:thetai_e2_power_expansion}, but the analysis involved is more subtle. 

\begin{proposition}
\label{prop:before_second_low_point_filtrations}
Suppose that $f\in \rmM_k$ with~$0<k<p$ has~$\omega_p(f)=k$ and let~$p - k + 1 < i < p$. Then we have the factor filtrations
\begin{gather*}
  \wtd\omega_{p^2}\big( \theta^i\, f \big)
=
\begin{cases}
  k + 2 i + (i+k) (p-1)
  \tx{,}\quad
&
  \tx{if } f \tx{ is ordinary at~$p$;}\\
  k + 2 i + (i+k-1) (p-1)
\tx{,}\quad
&
  \tx{if } f \tx{ is non-ordinary at~$p$.}
\end{cases}
\end{gather*}
\end{proposition}

\begin{proof}
Observe that the coefficients in the expansion~\eqref{eq:thetai_e2_power_expansion} of~$\theta^i\, f$ satisfy
\begin{gather}\label{eq:factdiv}
  p
\Bigisdivexact
  \mfrac{(i+k-1)!}{(j+k-1)!} 
\quad\tx{if }
  j < p - k + 1
\tx{.}
\end{gather}
Using this fact with~\eqref{eq:e2_power_congruence}, \eqref{eq:thetai_e2_power_expansion} and \eqref{eq:delta_e2_cong} gives the decomposition
\begin{gather*}
  \theta^i\, f
\equiv 
  p \td{f}_1
  +
  p \td{f}_2
  +
  \td{f}_3 \quad\pmod{p^2}
\end{gather*}
where
\begin{align*}
  \td{f}_1
&\defeq
  \sum_{j = 0}^{p-k}
  \mfrac{1}{12^{i-j}}\,
  \mbinom{i}{j}\, \mfrac{(i+k-1)!}{p (j+k-1)!}\,
  f_j \, E_{p+1}^{i-j}\,
  E_{p-1}^{(i-j)(2p-1)}
\tx{,}
\\
  \td{f}_2
&\defeq
  \sum_{j = p-k+1}^i
  \mfrac{i-j}{12^{i-j}}\,
  \mbinom{i}{j}\, \mfrac{(i+k-1)!}{(j+k-1)!}\,
  f_j \, E_{p+1}^{i-j+p-1}\,
  E_{p-1}^{(i-j-1)(2p-1) + p - 2}
\tx{,}
\\
  \td{f}_3
&\defeq
  \sum_{j = p-k+1}^i
  \mbinom{i}{j}\, \mfrac{(i+k-1)!}{(j+k-1)!}\,
  f_j \, \big(\partial E_{p-1} \big)^{i-j}\,
  E_{p-1}^{(i-j)(2p-1)}
\tx{.}
\end{align*}
For $\td{f}_3$ we have (recalling that $k<p$) the factor filtration bound
\begin{gather*}
  \wtd\omega_{p^2}(\td{f}_3)
\le
  \max \big\{
  k + 2j + (i-j)(p+1)
  \condsep
  p-k+1 \le j \le i
  \big\}
\leq
  k + 2i + (i-2) (p-1)
\tx{.}
\end{gather*}
For $p\td f_1$ we have 
\begin{gather*}
    \wtd\omega_{p^2}(p\td{f}_1)=\omega_p(\td{f}_1)\leq\max \big\{
  k + 2j + (i-j)(p+1)
  \condsep
  0 \le j \le p-k
  \big\}
=
  k + 2i + i (p-1)
\tx{.}
\end{gather*}

It remains to examine~$p \td{f}_2$. Recalling \eqref{eq:factdiv} and~\eqref{eq:thetai_e2_power_expansion}, we obtain the following congruence involving modular forms of different weights:
\begin{align*}
  \td{f}_2
&\equiv
  \sum_{j = 0}^i
  \mfrac{i-j}{12^{i-j}}\,
  \mbinom{i}{j}\, \mfrac{(i+k-1)!}{(j+k-1)!}\,
  f_j \, E_{p+1}^{i-j+p-1}
\\
&\equiv
  E_{p+1}^p\sum_{j = 0}^{i-1}
  \mfrac{i (i+k-1)}{12^{i-j}}\,
  \mbinom{i-1}{j}\, \mfrac{(i-1+k-1)!}{(j+k-1)!}\,
  f_j \, E_{p+1}^{i-1-j}
\equiv 
  \mfrac{i (i+k-1)}{12}\,
  E_{p+1}^p\,
  \theta^{i-1}\, f
  \quad\pmod{p}
\tx{.}
\end{align*}
Since~$p - k + 1 \le i-1 < p-1$, Proposition~\ref{prop:theta_cycle_modp} gives
\begin{gather*}
  \omega_p\big( \theta^{i-1}\, f \big)
=
\begin{cases}
  k + (i-1)(p+1) - (p-k+1)(p-1)
\tx{,}\quad
&
  \tx{if } f \tx{ is ordinary at~$p$;}
\\
  k + (i-1)(p+1) - (p-k+2)(p-1)
\tx{,}\quad
&
  \tx{if } f \tx{ is non-ordinary at~$p$.}
\end{cases}
\end{gather*}
Since the assumptions ensure that $i(i+k-1)\not\equiv 0 \;\pmod p$, we find using \eqref{eq:e2powermodp} that
\begin{gather*}
\wtd\omega_{p^2}(p \td{f}_2)
=
\omega_p(\td{f}_2)
=
  p (p+1)
  +
\omega_p\big( \theta^{i-1}\, f \big)\tx{.}
\end{gather*}
This gives the values
\begin{gather*}  \wtd\omega_{p^2}(p \td{f}_2)
=
\begin{cases}
  k + 2 i + (i+k)(p-1)
\tx{,}\quad
&
  \tx{if } f \tx{ is ordinary at~$p$;}
\\
  k + 2 i + (i+k-1) (p-1)
\tx{,}\quad
&
  \tx{if } f \tx{ is non-ordinary at~$p$.}
\end{cases}
\end{gather*}
Since these exceed the bounds which we obtained for the factor filtrations of~$p \td{f}_1$ and~$\td{f}_3$ we have $\wtd\omega_{p^2}\big( \theta^i\, f \big) = \wtd\omega_{p^2}(p \td{f}_2)$, which confirms the claimed factor filtrations.
\end{proof}

\begin{proposition}%
\label{prop:second_low_point_filtration}
Suppose that $f\in \rmM_k$ with $0<k<p$ has $\omega_p(f)=k$. At~$i = p$ we have the factor filtration
\begin{gather*}
  \wtd\omega_{p^2}\big( \theta^p\, f \big)
=
  k + 2 p + p (p-1)
\tx{.}
\end{gather*}
\end{proposition}

\begin{proof}
The coefficients in formula~\eqref{eq:thetai_e2_power_expansion} satisfy
\begin{align*}
  p^2
\Bigisdivexact
  \mbinom{p}{j}\,
  \mfrac{(p+k-1)!}{(j+k-1)!}
\tx{,}\qquad
&
  \tx{if } 1 \leq j \leq p-k\tx{;}
\\
  p
\Bigisdivexact
  \mbinom{p}{j}\,
  \mfrac{(p+k-1)!}{(j+k-1)!}
\tx{,}\qquad
&
  \tx{if } j = 0 \tx{ or } p-k+1 \leq j \leq p-1
\tx{.}
\end{align*}
It follows that for some $p$\nbd{}adic units~$\lambda_0$ and~$\lambda_j$ we have 
\begin{gather}
\label{eq:prop:second_low_point_filtration:thetapf}
  \theta^p\, f
\equiv
  p \lambda_0\, f_0 E_2^p
  +
  p \td{f}_1
  +
  f_p\pmod{p^2}
\tx{,}\quad\tx{where }
  \td{f}_1
\defeq
  \sum_{j=p-k+1}^{p-1}
  \lambda_j f_j\, E_2^{p-j}
\tx{.}
\end{gather}
We have~$\wtd\omega_{p^2}(f_p) \leq k+2p$.
For the terms appearing in~$\td{f}_1$ we have 
\begin{gather*}
  \omega_{p}\big( f_j\, E_2^{p-j} \big)
\leq
  k+2j+(p-j)(p+1)
=
  k+p(p+1)-j(p-1)
\tx{,}
\end{gather*}
from which 
\begin{gather*}
  \wtd\omega_{p^2}( p \td{f}_1 )
\le
  k + p(p+1) - (p-k+1)(p-1)
=
  k + 2p + (k-1)(p-1)
\tx{.}
\end{gather*}
Finally, we have
\begin{gather*}
\wtd\omega_{p^2}\big(
  p \lambda_0 f_0 \, E_2^p
  \big) 
=
\omega_p\big(
   f_0 \, E_2^p
  \big) 
  =
  k + p (p+1)
=
  k + 2p + p (p-1)
\tx{.}
\end{gather*}
Since this exceeds the other filtrations it determines the factor filtration of~$\wtd\omega_{p^2}(\theta^p\, f)$.
\end{proof}

\begin{proof}[Proof of Theorem~\ref{mainthm:first_p_interval_weight_filtration}]
Once the factor filtration of $\theta^if$ is known, the weight filtration is determined using \eqref{eq:filtcong} and 
Lemma~\ref{la:weight_from_factor_filtration}. For example, in the range~$0 < i < p - k + 1$ we have shown that the factor filtration 
of $\wtd\omega_{p^2}(\theta^i f)$
is
$k + 2i + (i + p + 1) (p - 1)$.  Since $p<i + p+ 1 < 2p$, we conclude that the weight filtration must be~$k + 2i + 2p (p - 1)$.
The analysis is similar in the other ranges and we omit the details.
\end{proof}

\section{Filtrations in the first part of each \texpdf{$p$}{p}-interval}%
\label{sec:first_half_p_interval}

In this section, we prove Theorem~\ref{mainthm:first_half_of_p_intervals_weight_filtration} and Corollaries~\ref{maincor:first_p_interval_weight_filtration_low_points} and~\ref{maincor:first_half_of_p_intervals_weight_filtration_low_points}.
We will always write 
\begin{gather}\label{eq:iprimedef}
  i=np+i'
\tx{,}\qquad
  0<n<p
\tx{,}\quad
  0<i'<p
\tx{.}
\end{gather}
Recall the definition \eqref{eq:def_exceptional} of an exceptional index $i$.
The results will follow from 
Theorem~\ref{thm:first_half_p_interval_low_point_filtrations}, which gives exact values and bounds for factor filtrations when $i'<p-k+1$.  

\begin{theorem}%
\label{thm:first_half_p_interval_low_point_filtrations}
Suppose that $f\in \rmM_k$ with~$0<k<p$ has
$\omega_p(f)=k$, that $i$ and $i'$ are as in~\eqref{eq:iprimedef}, and that $i'<p-k+1$.

\begin{enumeratearabic}
\item 
\label{it:thm:first_half_p_interval_low_point_filtrations:equality}
If\/~$i' \le (p - k + 1) - n$ and $i$ is not exceptional, then we have 
\begin{gather*}
  \wtd\omega_{p^2}\big( \theta^i\, f \big)
=
  k + 2 i + (i' + p - n + 1) (p-1)\tx{.}
\end{gather*}
(Note that this includes all non-exceptional $i$ with~$p < i < 2p - k + 1$.)

\item 
\label{it:thm:first_half_p_interval_low_point_filtrations:bound}
If\/~$i' > (p - k + 1) - n$ or if $i$ is exceptional, then we have the bounds
\begin{gather*}
  \td\omega_{p^2}\big( \theta^i\, f \big)
\le
\begin{cases}
  k + 2i + (i'+k-1) (p-1)
\tx{,}
\quad
  &\tx{if } n = 1
\tx{;}
\\
  k + 2i + (i'+k-2) (p-1)
\tx{,}
\quad
&\tx{if } n > 1
\tx{.}
\end{cases}
\end{gather*}
\end{enumeratearabic}
\end{theorem}

\begin{proof}
Analyzing the factorials (using for example Lucas' theorem for the binomial coefficients), we find that
\begin{gather}
\label{eq:prop:prf:first_half_p_interval_low_point_filtrations:binomial_factorial}
\begin{gathered}
  p \Bigisdiv \mfrac{(i+k-1)!}{(j + k - 1)!}
\tx{,}\quad
  \tx{if } j < np - k + 1
\tx{;}\qquad
  p^2 \Bigisdiv \mfrac{(i+k-1)!}{(j + k - 1)!}
\tx{,}\quad
  \tx{if } j < (n-1) p - k + 1
\tx{;}\\
  p \Bigisdiv \mbinom{i}{j}
\tx{,}\quad
  \tx{if } (n-1)p + i' < j < np
  \quad\tx{or}\quad (n-2)p + i' < j < (n-1)p
\tx{.}
\end{gathered}
\end{gather}
Together with~\eqref{eq:e2_power_congruence}, \eqref{eq:thetai_e2_power_expansion}, and the inequality $i'<p-k+1$, this gives
\begin{gather}
\label{eq:thetai_f1f2f3f4} 
  \theta^i\, f
\equiv
  p \td{f}_1
  +
  p \td{f}_2
  +
  p \td{f}_3
  +
  \td{f}_4
  \quad\pmod{p^2}
\end{gather}
with weight~$k + 2i + 2 (i-j) p (p-1)$ modular forms
\begin{align*}
  \td{f}_1
&\defeq
  \sum_{j = (n-1)p}^{(n-1)p+i'}
  \mfrac{1}{12^{i-j}}\,
  \mbinom{i}{j}\, \mfrac{(i+k-1)!}{p\,  (j+k-1)!}\,
  f_j\, E_{p+1}^{i-j}\;
  E_{p-1}^{(i-j)(2p-1)}
\tx{,}
\\
  \td{f}_2
&\defeq
  \sum_{j = np-k+1}^{np-1}
  \mfrac{1}{12^{i-j}}\,
  \mbinom{i}{j}\, \mfrac{(i+k-1)!}{p\, (j+k-1)!}\,
  f_j\, E_{p+1}^{i-j}\;
  E_{p-1}^{(i-j)(2p-1)}
\tx{,}
\\
  \td{f}_3
&\defeq
  \sum_{j = np}^i
  \mfrac{i-j}{12^{i-j}}\,
  \mbinom{i}{j}\, \mfrac{(i+k-1)!}{(j+k-1)!}\,
  f_j \, E_{p+1}^{i-j+p-1}\;E_{p-1}^{(i-j-1)(2p-1)+p-2}
\tx{,}
\\
  \td{f}_4
&\defeq
  \sum_{j = np}^i
  \mbinom{i}{j}\, \mfrac{(i+k-1)!}{(j+k-1)!}\,
  f_j \, \big( \partial E_{p-1} \big)^{i-j}\;
  E_{p-1}^{(i-j)(2p-1)}
\tx{.}
\end{align*}

For~$\td{f}_4$ we have the factor filtration bound
\begin{gather}
\label{eq:prop:prf:first_half_p_interval_low_point_filtrations:f4_bounds}
  \wtd\omega_{p^2}(\td{f}_4)
\le
  \max\big\{
  k + 2j + (i-j)(p+1)
  \condsep
  np \le j \le i
  \big\}
=
  k + 2i + i' (p-1)
\tx{.}
\end{gather}

Recall that~$\wtd\omega_{p^2}(p \td{f}_2)=\omega_p(\td{f}_2)$. 
We first estimate the contribution to the filtration from the term $j = np - k + 1$.
From~\eqref{eq:thetai_e2_power_expansion} we obtain
\begin{gather*}
  f_{np - k + 1}
\equiv
  \theta^{np - k + 1}\, f
\equiv
  \theta^{n-1 + p - k + 1}\, f
  \quad\pmod{p}
\tx{.}
\end{gather*}

If $0< n< k$ then we have $p-k+1\le n-1 + p - k + 1<p$, and Proposition~\ref{prop:theta_cycle_modp} gives (with $j = np - k + 1$)
\begin{gather}\label{eq:large_fj_est}
\begin{aligned}
  \omega_p\big(
  f_j\, E_{p+1}^{i-j}
  \big)
\le{}&
  k + (n - 1 + p-k+1)(p+1) - (p-k+1) (p-1) + (i-j) (p+1)
\\
={}&
  k + 2i + (i' +k- n)(p-1)
\tx{.}
\end{aligned}
\end{gather}
If however~$k\leq n<p$ we have~$f_{np - k + 1} \equiv \theta^{n - k + 1}\, f \,\pmod p$, so Proposition~\ref{prop:theta_cycle_modp} gives (again with $j = np - k + 1$)
\begin{gather*}
  \omega_p\big( f_j\, E_{p+1}^{i-j} \big) \le k+(n-k+1)(p+1)+(i-j)(p+1)
  \tx{,}
\end{gather*}
which is less than the bound in~\eqref{eq:large_fj_est}.
So the contribution from the term with $j=np-k+1$ is bounded by~\eqref{eq:large_fj_est}.

For the terms with~$np - k + 2 \le j \le np - 1$, each term contributing to~$\td{f}_2$ has 
\begin{gather}
\label{eq:second_est}
\begin{aligned}
  \omega_p\big(
  f_j\, E_{p+1}^{i-j}
  \big)
&\le
  k + 2j + (i-j) (p+1)
=
  k + 2i + (i-j) (p-1)
\\
&\le 
  k + 2i + (i'+k-2) (p-1)
\tx{.}
\end{aligned}
\end{gather}
The estimate~\eqref{eq:large_fj_est} is larger than the estimate~\eqref{eq:second_est} if and only if $n=1$.
Therefore we obtain 
\begin{gather}
\label{eq:prop:prf:first_half_p_interval_low_point_filtrations:f2_bounds}
  \td\omega_{p^2}(p \td{f}_2)
\le
\begin{cases}
  k + 2i + (i'+k-1) (p-1)
\tx{,}
\quad
&
  \tx{if } n = 1
\tx{;}
\\
  k + 2i + (i'+k-2) (p-1)
\tx{,}
\quad
&
  \tx{if } n > 1
\tx{.}
\end{cases}
\end{gather}

Next we inspect~$\td{f}_1$ and~$\td{f}_3$. We omit powers of~$E_{p-1}$ and hence work with congruences involving modular forms of different weights. We employ Wilson's theorem, Lucas's theorem, and the fact that $i'<p-k+1$ to find that
\begin{align*}
  \td{f}_1
&\equiv
  \sum_{j = (n-1)p}^{i-p}
  \mfrac{1}{12^{i-j}}\,
  \mbinom{i}{j}\, \mfrac{(i+k-1)!}{p\, (j+k-1)!}\,
  f_j\, E_{p+1}^{i-j}
\\
&\equiv
  \mfrac{-n}{12^p} E_{p+1}^p\,
  \sum_{j = (n-1)p}^{i-p}
  \mfrac{1}{12^{i-p-j}}\,
  \mbinom{i}{j}\, \mfrac{(i-p+k-1)!}{(j+k-1)!}\,
  f_j\, E_{p+1}^{i-p-j}
\\
&\equiv
  \mfrac{-n^2}{12^p} E_{p+1}^p\,
  \sum_{j = (n-1)p}^{i-p}
  \mfrac{1}{12^{i-p-j}}\,
  \mbinom{i-p}{j}\, \mfrac{(i-p+k-1)!}{(j+k-1)!}\,
  f_j\, E_{p+1}^{i-p-j}
  \quad\pmod{p}
  \tx{.}
\end{align*}
Arguing as in \eqref{eq:prop:prf:first_half_p_interval_low_point_filtrations:binomial_factorial}
and using \eqref{eq:thetai_e2_power_expansion} gives
\begin{gather*}
  \td{f}_1\equiv \mfrac{-n^2}{12} E_{p+1}^p\,
  \theta^{i-p}\, f
\equiv
  \mfrac{-n^2}{12} E_{p+1}^p\,
  \theta^{i'+n-1}\, f
  \quad\pmod{p}\tx{.}
\end{gather*}
For~$\td{f}_3$ we observe that there is no contribution from~$j = i$ and use the identity~$(i-j) \binom{i}{j} = i \binom{i-1}{j}$  to obtain
\begin{align*}
  \td{f}_3
& \equiv
  \mfrac{i(i+k-1)}{12}\,
  E_{p+1}^p\,
  \sum_{j = np}^{i-1}
  \mfrac{1}{12^{i-1-j}}\,
  \mbinom{i-1}{j}\, \mfrac{(i-1+k-1)!}{(j+k-1)!}\,
  f_j \, E_{p+1}^{i-1-j}
\\
&\equiv
  \mfrac{i(i+k-1)}{12} E_{p+1}^p\,
  \theta^{i-1}\, f
\equiv
  \mfrac{i(i+k-1)}{12} E_{p+1}^p\,
  \theta^{i'+n-1}\, f
  \quad\pmod{p}
\tx{.}
\end{align*}
We conclude that
\begin{gather}\label{eq:prop:prf:first_half_p_interval_low_point_filtrations:f1_f3_bounds}
    \begin{aligned}
  \td{f}_1 + \td{f}_3
&\equiv
  \mfrac{-n^2}{12} E_{p+1}^p\,
  \theta^{i'+n-1}\, f
  +
  \mfrac{i(i+k-1)}{12} E_{p+1}^p\,
  \theta^{i'+n-1}\, f
\\
&\equiv
  \tfrac{1}{12}
  \big( i^2 + (k-1) i - n^2 \big)\,
  E_{p+1}^p\,
  \theta^{i'+n-1}\, f
  \quad\pmod{p}
\tx{.}
\end{aligned}
\end{gather}

From \eqref{eq:prop:prf:first_half_p_interval_low_point_filtrations:f2_bounds} and \eqref{eq:prop:prf:first_half_p_interval_low_point_filtrations:f4_bounds}  we have
\begin{align}
\label{eq:f2f4filt}
  \wtd{\omega}_{p^2}\big( p \td{f}_2 + \td{f}_4 \big)
  \le{}
\begin{cases}
  k + 2i + (i'+k-1) (p-1)
\tx{,}
\quad
&
  \tx{if } n = 1
\tx{;}
\\
  k + 2i + (i'+k-2) (p-1)
\tx{,}
\quad
&
  \tx{if } n > 1
\tx{.}
\end{cases}
\end{align}
If $i$ is exceptional (i.e. $i^2 + (k-1) i - n^2 \equiv 0 \,\pmod{p}$)
then the term \eqref{eq:prop:prf:first_half_p_interval_low_point_filtrations:f1_f3_bounds} vanishes modulo $p$.  
Together with~\eqref{eq:thetai_f1f2f3f4}  and \eqref{eq:f2f4filt}  this
gives the claimed bounds in part~\ref{it:thm:first_half_p_interval_low_point_filtrations:bound} of the theorem.

Now assume that~$i$ is not exceptional. If~$f$ is ordinary at~$p$, then Proposition~\ref{prop:theta_cycle_modp} gives
\begin{gather*}
  \omega_p\big( \theta^{i'+n-1}\, f \big)=
   \begin{cases}
  k + 2i + (i'-n-1)(p-1) - 2
\tx{,}\quad
&
  \tx{if } i' + n - 1 < p - k + 1
\tx{;}
\\
   k + 2i + (i'-n-p+k-2)(p-1) - 2
\tx{,}\quad
&
  \tx{if } p - k + 1 \le i' + n - 1 < p
\tx{.}
\end{cases}
\end{gather*}
Furthermore, when $p \le i' + n - 1$  
the same proposition (note that $i'+n-p<p-k+1$)
gives
 \begin{gather*}
 \omega_p\big( \theta^{i'+n-1}\, f \big)=\omega_p\big( \theta^{i'+n-p}\, f \big)
=  k + 2i + (i'-n-p-2)(p-1) - 2\tx{.}
\end{gather*}
Adding~$p (p + 1) = (p+2)(p-1) + 2$ in each case gives
\begin{gather*}
  \wtd\omega_{p^2}\left(p(\td{f}_1 + \td{f}_3)\right)=
  \begin{cases}
  k + 2i + (i'+p-n+1)(p-1)
\tx{,}\quad
&
  \tx{if } i' + n -1 < p - k + 1
\tx{;}
\\
  k + 2i + (i'-n+k)(p-1)
\tx{,}\quad
&
  \tx{if } p - k + 1 \le i' + n - 1 < p
\tx{;}
\\
  k + 2i + (i'-n)(p-1)
\tx{,}\quad
&
  \tx{if } p \le i' + n - 1 
\tx{.}
\end{cases}
\end{gather*}
Similarly, if~$f$ is non-ordinary at~$p$, then we obtain
\begin{gather*}
  \wtd\omega_{p^2}\left(p(\td{f}_1 + \td{f}_3)\right)=
  \begin{cases}
  k + 2i + (i'+p-n+1)(p-1)
\tx{,}\quad
&
  \tx{if } i' + n -1 < p - k + 1
\tx{;}
\\
  k + 2i + (i'-n+k-1)(p-1)
\tx{,}\quad
&
  \tx{if } p - k + 1 \le i' + n - 1 < p-1
\tx{;}
\\
  k + 2i + (i'-n)(p-1)
\tx{,}\quad
&
  \tx{if } p-1 \le i' + n - 1 
\tx{.}
\end{cases}
\end{gather*}

If~$i' + n - 1 < p - k + 1$ then $i'+k-1<p-n+1$ and the bound in \eqref{eq:f2f4filt} is strictly less 
than the values above, 
and we find that 
\begin{gather*}
  \wtd{\omega}_{p^2}\big( \theta^i\, f \big)
=
  \wtd\omega_{p^2}\big( p( \td{f}_1 +  \td{f}_3) \big)\,
=
  k + 2i + (i'+p-n+1)(p-1)
\tx{,}
\end{gather*}
which confirms part~\ref{it:thm:first_half_p_interval_low_point_filtrations:equality} of the theorem.

It remains to consider the cases when 
$i$ is not exceptional and
$i' + n - 1 \geq p - k + 1$.
In these cases the factor filtration of~$p(\td{f}_1 + \td{f}_3)$ does not dominate the factor filtration in 
\eqref{eq:f2f4filt},
and we can deduce only that
\begin{gather*}
  \td\omega_{p^2}\big( \theta^i\, f \big)
\le
\begin{cases}
  k + 2i + (i'+k-1) (p-1)
\tx{,}
\quad
&
  \tx{if } n = 1
\tx{;}
\\
  k + 2i + (i'+k-2) (p-1)
\tx{,}
\quad
&
  \tx{if } n > 1
\tx{.}
\end{cases}
\end{gather*}
This completes the proof of the theorem.
\end{proof}

\begin{proof}[Proof of Theorem~\ref{mainthm:first_half_of_p_intervals_weight_filtration}]
As mentioned above,the statements when $i'=0$ and $i'=p-k+1$ are a consequence of~\cite[Theorem 1.8]{kim-lee-2023}. 
For the rest we use Theorem~\ref{thm:first_half_p_interval_low_point_filtrations}, Lemma~\ref{la:weight_from_factor_filtration}, and \eqref{eq:filtcong}.

By assumption we have $4<i'+k<p+1$. In the cases covered by part~\ref{it:thm:first_half_p_interval_low_point_filtrations:bound} of Theorem~\ref{thm:first_half_p_interval_low_point_filtrations} this gives the weight filtration bound $k+2i+p(p-1)$.
In the cases covered by part~\ref{it:thm:first_half_p_interval_low_point_filtrations:equality}
we see that 
\begin{align*}
  0<i'+p-n+1<p+1
\tx{,}
&\quad
  \tx{if $0<i'<n$ and $i'\leq p-k+1-n$}
\tx{;}
\\
  p<i'+p-n+1<2p
\tx{,}
&\quad
  \tx{if $n\leq i'\leq p-k+1-n$}
\tx{,}
\end{align*}
from which the exact weight filtrations follow.
\end{proof}

\begin{proof}[Proof of Corollary~\ref{maincor:first_p_interval_weight_filtration_low_points}]
To show that the second low point occurs at index $p$ 
we use Theorems~\ref{mainthm:first_p_interval_weight_filtration} and~\ref{mainthm:first_half_of_p_intervals_weight_filtration} together with the fact that~$p+1$ is not an exceptional index.
All other assertions follow immediately from 
Theorem~\ref{mainthm:first_p_interval_weight_filtration}. 
\end{proof}

\begin{proof}[Proof of Corollary~\ref{maincor:first_half_of_p_intervals_weight_filtration_low_points}]
Part~\ref{it:maincor:first_half_of_p_intervals_weight_filtration_low_points:rises} follows from Theorem~\ref{mainthm:first_half_of_p_intervals_weight_filtration} since the hypotheses ensure that~$i$ and~$i+1$ are both contained in one of the two ranges where the filtrations are exactly determined.
For part~\ref{it:maincor:first_half_of_p_intervals_weight_filtration_low_points:non_low_points} 
note that under the 
hypotheses the filtration at~$i$ is exactly determined by the theorem,
while the filtration at~$i-1$ is bounded to be less than the filtration at~$i$.

 To prove part~\ref{it:maincor:first_half_of_p_intervals_weight_filtration_low_points:regular_low_points},  
let~$i =  np+p - k + 2 - n$ be as in the statement.
Using both parts of Theorem~\ref{mainthm:first_half_of_p_intervals_weight_filtration} we see that we always have $\omega_{p^2}(\theta^i f)\leq k + 2i + p(p-1)$.
Since $2n \le p-k+1$ by hypothesis and~$k \ge 4$ it follows that 
\begin{gather}\label{eq:theta_i_upper_bound}
    k + 2i - p(p-1)\leq (4-k)(p+1)-2n<0
    \tx\,
\end{gather}    
    from which we conclude using \eqref{eq:filtcong} that 
\begin{gather*}
    \omega_{p^2}(\theta^i f)=k+2i\quad\text{or}\quad k + 2i + p(p-1)\tx{.}
\end{gather*}
The index $i-1$ is not exceptional since
\begin{gather*}
  (i-1)^2 + (k-1)(i-1) - n^2 \equiv n(k-1)\not\equiv 0
  \quad\pmod{p}
\tx{.}
\end{gather*}
Since the hypotheses give
$n \leq i'-1 = p-k+1-n$,
 it follows from Theorem~\ref{mainthm:first_half_of_p_intervals_weight_filtration} that 
\begin{gather*}
 \omega_{p^2}(\theta^{i-1} f) = k + 2(i-1) + 2p(p-1)
\tx{.}
\end{gather*}
From \eqref{eq:theta_i_upper_bound} we have 
$k+2(i+1)-p(p-1)\leq 0$.   Note that $\omega_{p^2}(\theta^{i+1}\, f)\neq0$ (equality would imply that $\theta^{i+2}\, f\equiv 0\pmod p$ which contradicts Proposition~\ref{prop:theta_cycle_modp}).
It follows that 
\begin{gather*}
  \omega_{p^2}(\theta^{i+1} f)
\geq
  k+2(i+1)
\tx{.}
\end{gather*}

If $\omega_{p^2}(\theta^{i} \,f)=k+2i$ or $\omega_{p^2}(\theta^{i+1}\, f)\geq k + 2(i+1) + p(p-1)$ then there is a low point at $i$ as claimed.
In the remaining case we have $\omega_{p^2}(\theta^{i} \,f)=k+2i+p(p-1)$ and $\omega_{p^2}(\theta^{i+1}\, f)= k + 2(i+1)$.
By \eqref{eq:theta_i_upper_bound} we have 
$k+2(i+2)-p(p-1)\leq 2$.
Since we cannot have
$\omega_{p^2}(\theta^{i+2}\, f)=0$ or $2$, it follows that 
$\omega_{p^2}(\theta^{i+2} f)\geq k+2(i+2)$, which establishes that there is a low point at $i+1$ in this case.

We turn to the proof of part~\ref{it:maincor:first_half_of_p_intervals_weight_filtration_low_points:exceptional_low_points}. We must show that there is a low point at $i$ whenever $i$ is an exceptional index and 
$n \le i' \le p - k + 1 - n$. It is straightforward to check that indices with~$i'=n$ or~$i'=p-k+1-n$ are not exceptional.
It follows that
\begin{gather*}
    n\leq i'-1< i'+1\leq p-k+1-n
    \tx{.}
\end{gather*}    
The filtration values and bounds in Theorem~\ref{mainthm:first_half_of_p_intervals_weight_filtration} show that there is a low point at $i$  unless one of $i-1$ or $i+1$ is exceptional and it
therefore suffices to show that there are no consecutive solutions~$i'$, $i'+1$ to the congruence \eqref{eq:def_exceptional}.  This follows since 
\begin{gather*}
  \big( (i'+1)^2 + (k-1)(i'+1) - n^2 \big)
-
  \big( i'^2 + (k-1)i' - n^2 \big)
=
  2i'+ k
\tx{,}
\end{gather*}
and by the assumptions the right side is positive, even, and less than~$2p$.
\end{proof}

\begin{remark}
\label{rm:result_range}
Theorem~\ref{mainthm:first_p_interval_weight_filtration} and the results in Section~\ref{sec:first_p_interval} provide exact values for the weight and factor filtration of~$\theta^i\, f$ at~$p$ positions.
 Theorems~\ref{mainthm:first_half_of_p_intervals_weight_filtration}
and~\ref{thm:first_half_p_interval_low_point_filtrations} provide exact values for both filtrations at non-exceptional indices when~$0 < i' \le p - k + 1 - n$.  Since in each $p$-interval there are at most two exceptional indices, 
we have exact values for both filtrations of~$\theta^i\, f$ with~$0 \le i \le p(p-1)$ for at least
\begin{gather*}
   p + \sum_{n=1}^{p-k+1} (p-k-1-n)
=p+\mfrac{(p-k-4)(p-k+1)}2
\end{gather*}
positions. This is asymptotically~50\% of the~$p(p-1)$ total number of positions. Similarly, we have bounds for at least
\begin{gather*}
  p
  +
  \sum_{n = 1}^{p-1}
  (p-k+1)
=
  p
  +
  (p-k+1)(p-1)
\end{gather*}
positions, which yields asymptotically~100\%.

\end{remark}

\vspace{1.5\baselineskip}
\phantomsection
\addcontentsline{toc}{section}{References}
\markright{References}
\label{sec:references}
\sloppy
\printbibliography[heading=none]

\filbreak
\Needspace*{5\baselineskip}
\noindent%
\rule{\textwidth}{0.15em}
\\\nopagebreak

{\small\noindent
Scott Ahlgren\\\nopagebreak
Department of Mathematics\\\nopagebreak
University of Illinois\\\nopagebreak
Urbana, IL 61801, USA\\\nopagebreak
E-mail: \url{sahlgren@illinois.edu}
}\vspace{.5\baselineskip}

{\small\noindent
Martin Raum\\\nopagebreak
Chalmers tekniska högskola och G\"oteborgs Universitet\\\nopagebreak
Institutionen f\"or Matematiska vetenskaper\\\nopagebreak
SE-412 96 G\"oteborg, Sweden\\\nopagebreak
E-mail: \url{martin@raum-brothers.eu}\\\nopagebreak
Homepage: \url{https://martin.raum-brothers.eu}
}\vspace{.5\baselineskip}

{\small\noindent
Olav K. Richter\\\nopagebreak
Department of Mathematics\\\nopagebreak
University of North Texas\\\nopagebreak
Denton, TX 76203, USA\\\nopagebreak
E-mail: \url{richter@unt.edu}
}%

\ifdraft{%
\listoftodos%
}

\end{document}

